\documentclass[leqno]{article} 

\usepackage{graphicx}
\usepackage{authblk} 

\usepackage{amsmath}
\usepackage{amssymb}
\usepackage{amsthm}{\tiny }
\usepackage[ruled,linesnumbered]{algorithm2e}
\usepackage{geometry}[1 inch]

\usepackage{color} 

\graphicspath{{figs/}}

\usepackage{wrapfig}

\newtheorem{theorem}{Theorem}

\newtheorem{definition}{Definition}
\newtheorem{proposition}{Proposition}

\newcommand{\radius}{r}

\newcommand{\RR}{\mathbb{R}}
\newcommand{\NN}{\mathbb{N}}
\newcommand{\K}{\mathcal{K}}

\newcommand{\C}{\mathcal{C}}

\renewcommand{\L}{\mathcal{L}}

\newcommand{\qlb}{\mathrm{qlb}}

\newcommand{\ub}{\mathrm{ub}}
\newcommand{\lb}{\mathrm{lb}}
\newcommand{\Y}{\mathcal{Y}}
\newcommand{\qlbr}{\mathbf{qlb}}
\newcommand{\lbr}{\mathbf{lb}}
\newcommand{\xr}{\mathbf{x}}

\newcommand{\xbest}{x_{\mathrm{best}}}
\newcommand{\epsNewton}{\epsilon_{\mathrm{Newton}}}

\renewcommand{\int}{\mathrm{int}}

\newcommand{\lambdamin}{\lambda_{\mathrm{min}}}
\newcommand{\lambdamax}{\lambda_{\mathrm{max}}}

\newcommand{\threepartdef}[5]
{
	\left\{
	\begin{array}{ll}
		#1 & \mbox{if } #2 \\
		#3 & \mbox{if } #4 \\
		#5 & \mbox{otherwise}
	\end{array}
	\right.
}

\title{ Quasi Branch and Bound for Smooth Global Optimization}
\author{Nadav Dym}
\affil{Duke University}
\date{} 

\begin{document}

\maketitle
\begin{abstract}
Quasi branch and bound is a recently introduced generalization of branch and bound, where lower bounds are replaced by a relaxed notion of quasi-lower bounds, required to be lower bounds only for sub-cubes containing a minimizer. This paper is devoted to studying the possible benefits of this approach, for the problem of minimizing a smooth function over a cube. This is accomplished by suggesting two quasi branch and bound algorithms, qBnB(2) and qBnB(3), that compare favorably with alternative branch and bound algorithms.

The first algorithm we propose, qBnB(2), achieves second order convergence based  only on a bound on second derivatives, without requiring calculation of derivatives. As such, this algorithm is suitable for derivative free optimization, for which typical algorithms such as Lipschitz optimization only have first order convergence and so suffer from limited accuracy due to the clustering problem. Additionally, qBnB(2) is provably more efficient than the second order Lipschitz gradient algorithm which does require exact calculation of gradients.

The second algorithm we propose, qBnB(3), has third order convergence and finite termination. In contrast with BnB algorithms with similar guarantees who typically compute lower bounds via solving relatively time consuming convex optimization problems, calculation of qBnB(3) bounds only requires solving a small number of Newton iterations. Our experiments verify the potential of both these methods in comparison with state of the art branch and bound algorithms.
\end{abstract}
\section{Introduction}
We consider the problem of optimizing a smooth function over a $d$-dimensional cube. Our focus is on guaranteed global optimization of such functions, a task which is typically addressed using \emph{Branch and Bound (BnB)} algorithms.  BnB algorithms have  many applications in science,engineering, economics and other fields. Examples can be found in surveys such as \cite{audet2017derivative, horst2013handbook,hare2013survey}. There are also applications for low dimensional problems in computer vision \cite{campbell2017globally,yang2015go,hartley2009global} which seem to be less well known in the general global optimization community. BnB algorithms are especially suitable for low dimensional problems, having many local minima, and where accuracy is of essence. For high dimensional problems local optimization algorithms will typically be preferable as the worst case complexity of BnB algorithms is exponential in the dimension, a problem which seems unavoidable as optimization over a $d$-dimensional cube is NP hard \cite{kreinovich2005beyond}.

 The notion of \emph{quasi-branch and bound (qBnB)} algorithms was introduced in \cite{dym2019linearly},  in the context of the rigid alignment problem. After suggesting qBnB as a general principle, the authors suggested a  qBnB algorithm tailored for the structure of the rigid alignment problem, and demonstrated that it is considerably more efficient than competing BnB algorithms suggested for this problem (see Appendix~\ref{app:conditional_second_order} for more details). The goal of this paper is to develop the concept of qBnB further, in the more general context of optimization of a smooth function over a cube. 
 

\begin{wrapfigure}[9]{r}{0.4\columnwidth}
	\vspace{-2 em}
	\centering
	\includegraphics[width=0.35\columnwidth]{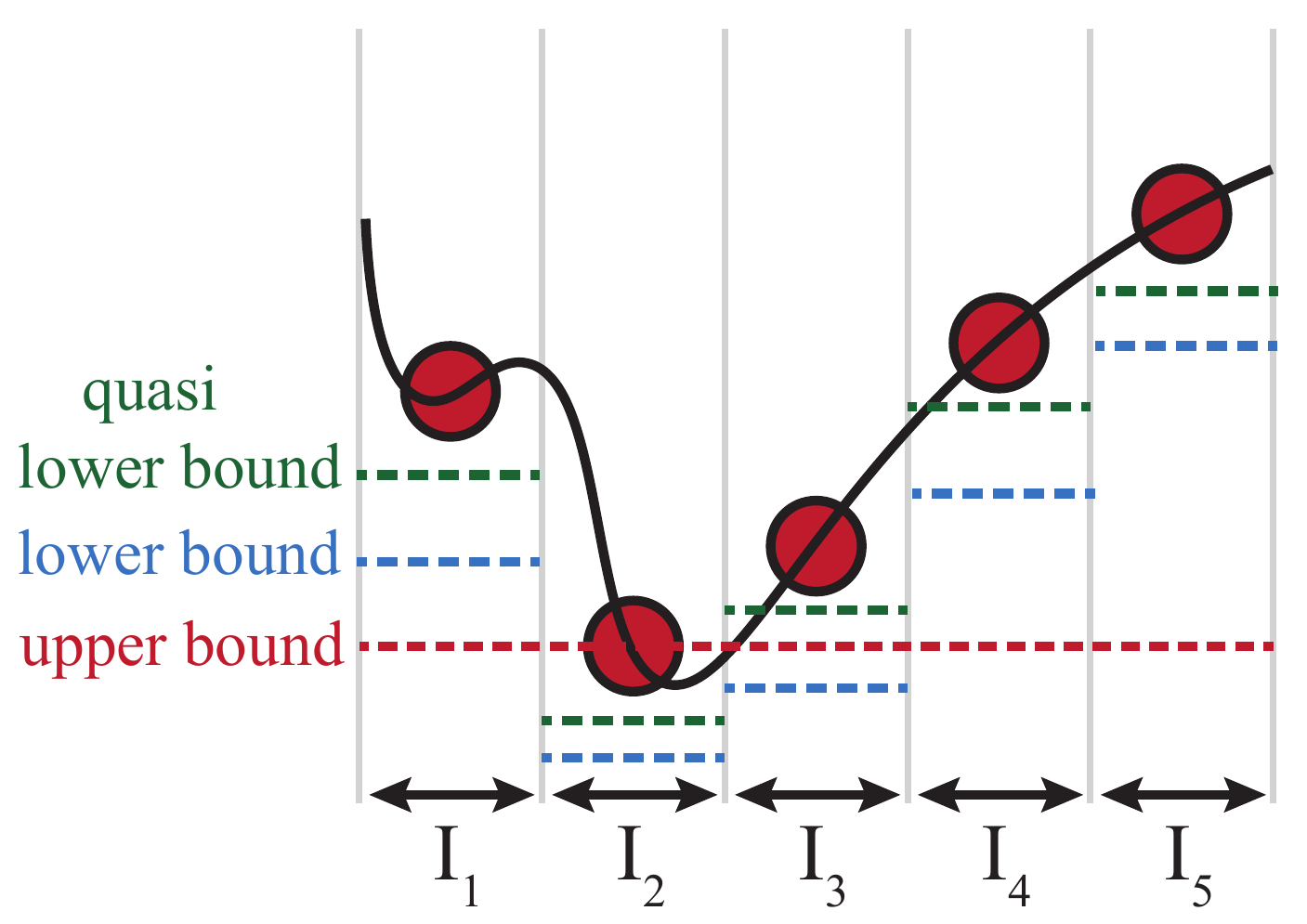}
	\vspace{-1 em}
	\caption{qBnB illustration}
	\label{fig:quasi}
\end{wrapfigure}
The basic idea for qBnB is illustrated in Figure~\ref{fig:quasi}. The figure illustrates optimization of a 1-dimensional function, whose domain is partitioned into five intervals $I_1,\ldots,I_5 $.  BnB algorithms compute a lower bound $\lbr(I_j)$ per interval, and compare it with an upper bound $\ub$ (denoted by a long red line) obtained from samplings of the function. Intervals $I_j$ satisfying
$\lbr(I_j)>\ub $
cannot contain a minimizer and so can be safely eliminated from the searching procedure. In Figure~\ref{fig:quasi} the intervals $I_1,I_4$ and $I_5$ would be eliminated by the BnB algorithm. 

Quasi-BnB  algorithms replace lower bounds with quasi-lower bounds. These are defined to be an assignment $\qlbr(I_j) $ of a scalar value per interval, which is required to be a lower bound for intervals $I_j$  \emph{containing a minimizer}, but may not be a lower bound for other intervals. For example, the green lines in Figure~\ref{fig:quasi}  are not lower bounds for the intervals $I_3$ and $I_4$, but are lower bounds for the interval $I_2$ which contains the minimizer of the function, and so they are valid quasi-lower bounds. 
Note that, as with  true lower bounds, intervals $I_j$ for which 
$
\qlbr(I_j)>\ub 
$
cannot contain a lower bound, and so can be safely removed from the searching procedure. In Figure~\ref{fig:quasi} this removal criterion  would eliminate all  intervals except for the interval $I_2$ which contains the minimizer.

Quasi-lower bounds are a generalization of lower bounds, as by definition lower bounds are also quasi-lower bounds. At first glance, it may not be clear how useful this generalization can be, as it is not generally possible to know whether a given partition element contains a minimizer or not. Nonetheless we find that the notion of quasi-lower bounds is useful, as it enables us to `assume by contradiction' that every cube  contains a minimizer, and so a point with vanishing gradient and positive semi-definite Hessian. These properties are useful for arriving at tighter bounds than a standard lower bound procedure would deliver.

The first algorithm we suggest in this paper, qBnB(2), was already mentioned in passing in the original qBnB paper \cite{dym2019linearly}. For unconstrained optimization problems, minimized in the interior of the cube, it utilizes the vanishing of the gradient at minimizers to propose a second order algorithm which does not require computation of derivatives, but only  the ability to evaluate the function and bound second derivatives. As such, qBnB(2) is a promising alternative to the first order Lipschitz algorithm for derivative free optimization (\cite{audet2017derivative,larson2019derivative,conn2009introduction}), where the functions minimized are smooth but their derivatives are not accessible. We also show qBnB(2) can be competitive when derivatives are available, by showing both theoretically and empirically that it is more accurate than the second order  "Lipschitz gradient" BnB algorithm \cite{kvasov2012lipschitz,kvasov2009univariate}. Finally, we  show how to modify qBnB(2) for global constrained minimization on a cube, naming the resulting algorithm constrained-qBnB(2). This modified algorithm also does not require calculation of derivatives, and for unconstrained problems its timing is comparable to unconstrained qBnB(2).

The second algorithm we present in this paper utilizes the positive-definiteness of the Hessian at a minimizer to obtain a third order algorithm we name qBnB(3). For non-degenerate problems, this algorithm enjoys the finite termination property- it can find the global minimum exactly in a finite number of iterations. In comparison with the popular $\alpha$BB algorithm \cite{adjiman1998global,adjiman1998global,androulakis1995alphabb} which also enjoys the finite termination property, the major advantage of qBnB(3) is that it only employs a small number of Newton iterations to compute the quasi lower bound, while $\alpha$BB computes lower bounds by the more time consuming process of optimizing general box-constrained convex programs. 

In practice we find that the bounding procedure in qBnB(3) is more accurate than qBnB(2) for small cubes, but is often less accurate for large cubes. This motivates a combined algorithm which uses a well informed principled criterion to choose a bounding procedure  based on cube size and problem parameters. Our experiments show this algorithm, which we name qBnB(2+3), outperforms all other qBnB and BnB algorithms described in this paper. 

\textbf{Paper organization}
In Section~\ref{sec:background} we fix problem notation, and review common BnB algorithms for continuous optimization on a cube. In Section~\ref{sec:qBnB} we introduce the idea of quasi-BnB algorithms formally. In Section~\ref{sec:second_order} we discuss the second order algorithms qBnB(2) and constrained-qBnB(2), and in  Section~\ref{sec:third_order} we discuss the third order algorithms qBnB(3). Experimental results are described in Section~\ref{sec:results}. 



\section{Background}\label{sec:background}
\subsection{Problem Definition and Notation}

A cube in $\RR^d$ is a set $\C$ of the form
$$\C=\{x \in \RR^d| \, |x_i-a_i|\leq h_i, \ \forall i, i=1,\ldots ,d \} .$$
We call $a$  the center of the cube and denote it by $x_\C$. We call $h$ the half-edge length of the cube, and $\|h\|_2$ the radius of the cube. We denote the radius of the cube $\C$ by $r(\C)$.  
 
Let $\C_0$ be a cube in $\RR^d$ and $f:\C_0 \to \RR $ a continuous function. In this paper we consider the problem of globally minimizing
\begin{equation}\label{eq:problem}
\min_{x \in \C_0} f(x)
\end{equation} 
We will denote the minimum of \eqref{eq:problem} by $f_*$ and we will typically use $x_*$ to denote a minimizer. For $\epsilon>0$ we say that $x \in \C_0$ is an $\epsilon$-optimal solution of \eqref{eq:problem} if $f(x)-f_*<\epsilon $.  We denote the collection of all sub-cubes of $\C_0$ by $\K$. For $\C \in \K$ we denote the minimum of $f$ on $\C$ by $f_*(\C)$.

We say that $(f,\C_0)$ is an \emph{unconstrained} optimization problem if there exists an open set $U_0$ containing $\C_0$ such that 
\begin{equation}\label{eq:unconstrained}
\inf_{x\in U_0} f(x)=\min_{x \in \C_0} f(x).
\end{equation}

We say that $f \in C^k(\C_0) $ if $f$ has $k$ derivatives in an open set containing $\C_0$. We denote the Hessian and gradient (if they exist) of $f$  at a point $x$ by $g(x)$ and $H(x)$. We denote the minimal and maximal eigenvalue of $H(x)$ by $\lambdamin(x) $ and $\lambdamax(x)$. For $B \subseteq \C_0$ we say that the gradient of $f \in C^1(\C_0)$ is Lipschitz in $B$  with Lipschitz constant $L_2\geq 0$ if the function $x\in B \mapsto g(x)$ is $L_2$ Lipschitz, and for $f \in C^2(\C_0) $ we say that the hessian is Lipschitz in $B$ with Lipschitz constant $L_3$ if the function $x\in B \mapsto H(x)$ is Lipschitz, where the norm on $H(x)$ is taken to be the operator norm. 

We say that $(f,\C_0)$ is \emph{non-degenerate}, if it is unconstrained, $f$ is in $C^2(\C_0) $, there are a finite number of minimizers, and the Hessian at each minimizer is strictly positive definite.  
\subsection{Branch and bound algorithms}\label{subsec:BnB}

Branch and bound (BnB) algorithms are able to guarantee $\epsilon$-optimal solutions for problems of the form \eqref{eq:problem}, by using a coarse to fine search procedure.
The search procedure aims at eliminating cubes which do not contain minimizers using lower bound and sampling rules, which are defined as follows:   


 \begin{definition}[lower bound and sampling rules]\label{def:lbr}
	We say that $\xr: \K \to \C_0 $ is a sampling rule, if
	$$\xr(\C)\in \C, \;  \forall \C \in \K $$
	We say that $\lb: \K \to [-\infty,\infty) $ is a lower  bound rule for minimizing $f$ over $\C_0$, if
	$$\lbr(\C) \leq f_*(\C), \forall \C \in \K $$
\end{definition}

In this section we describe several popular BnB algorithms for minimizing \eqref{eq:problem}, and discuss their relative advantages and disadvantages. We focus on methods for computing lower bound and sampling rules, and not on other algorithmic aspects such as the order in which the cubes are searched and refined. 

We begin our discussion with the classical Lipschitz optimization \cite{ahmed2020combining,li20073d,jones1993lipschitzian} algorithm. This algorithm assumes that $f$ is Lipschitz continuous on $\C_0 $ and a (not necessarily optimal) Lipschitz constant $L_1$ is known. In this case the sampling and lower bounds rules are selected to be
$$\xr(\C)=x_\C \text{ and } \lbr(\C)=f(x_\C)-L_1r  ,$$
where $r$ is the radius of $\C$ and $x_\C$ is the center of $\C$.  A significant disadvantage of Lipschitz optimization is that it can be very computational expensive to guarantee a high quality solution. This problem is caused by failure to rule out solutions which are close to optimal solutions, so that the search space of the Lipschitz algorithm in late stages of the algorithm typically contains a large cluster of sub-cubes around the solution which are never eliminated. Analysis \cite{wechsung2014cluster,neumaier2004complete,du1994cluster} of the \emph{clustering problem}, as it is known  in the global optimization literature , revealed that it is strongly related to the convergence order of the algorithm, which is defined as follows:
 \begin{definition}[convergence order]
 	Let $\xr$ and $\lbr$ be sampling and quasi lower bound rules for minimizing $f$ over $C_0 $. We say that $(\xr,\lbr) $ have convergence  order $\alpha$ for some $\alpha>0$, if there exists some $c>0$ such that for all cubes $\C\in \K$ with radius $r$,
 	\begin{equation}\label{eq:order}
 	f(\xr(\C))-\lbr(\C)\leq c\radius^{\alpha} .
 		\end{equation}
 \end{definition}

The Lipschitz algorithm has convergence order 1. In general, algorithms with convergence order $1 $ will encounter the clustering problem, in the sense that the complexity of obtaining an $\epsilon$ optimal solution will be polynomial in $1/\epsilon$. Algorithms with convergence order $2$ will generally avoid the clustering problem, in the sense that the complexity of obtaining an $\epsilon$ optimal solution will be proportional to a constant multiplied by $\log(1/\epsilon)$. However, when the problem is badly conditioned this constant can be very large. The asymptotic complexity of algorithms with convergence order $3$ is independent of the conditioning of the problem.
 
To achieve convergence order $2$ it is  typically necessary to assume that $f$ has a Lipschitz continuous gradient with Lipschitz constant $L_2$. For example an algorithm we will call the " Lipschitz gradient" algorithm \cite{kvasov2012lipschitz,kvasov2009univariate}  uses the observation that for a cube $\C \in \K $ with center $x_\C$ and radius $\radius$,
\begin{equation}\label{eq:lip_grad}
f(x)\geq f(x_\C)+g^T(x_\C)(x-x_\C)-\frac{L_2}{2} \radius^2, \quad \forall x \in \C\end{equation}
Accordingly $\lbr(\C)$ is defined as the minimum of the linear function which bounds $f$ from below,
\begin{equation}\label{eq:lip_grad_rule}
\lbr(\C)=\min_{x \in \C}  f(x_\C)+g^T(x_\C)(x-x_\C)-\frac{L_2}{2} \radius^2
\end{equation}
 and $\xr(\C) $ is chosen to be a minimizer of this linear function. These rules are simple to compute since the minimizer of the linear function on a cube is determined simply from the signs of the gradient.   We note that when $f$ is twice differentiable $L_2$ can be obtained as a bound on the spectral norm of the Hessian $H(x), x\in \C $. In this case  it is also possible to obtain a lower bound by replacing $L_2$ in \eqref{eq:lip_grad} with a lower bound for the minimal eigenvalue of $H(x), x\in \C$, as suggested, e.g., in  \cite{evtushenko2013deterministic}. 

Among the most popular second order algorithms is the $\alpha$BB algorithm \cite{adjiman1998global,adjiman1998global,androulakis1995alphabb}. This algorithm uses the following lower bounding rule: for  a cube $\C \in \K $ with midpoint $x_\C$ and  half edge length $h$, let $x_u=x_\C+ h$ and $x_l=x_\C-h $.  Then for $\alpha\geq 0$
$$f(x)\geq \ell_{\alpha}(x) \equiv f(x_\C)+g^T(x_\C)(x-x_\C)+\alpha(x-x_u)^T(x-x_l) , \quad \forall x \in \C$$  
If $m$ is a lower bound for the minimal eigenvalue of $H(x), x \in \C $ and $\alpha=\max \{0,-m/2\} $, $\ell_\alpha$ is convex and so optimizing $\ell_\alpha$ over $\C$ is tractable. According $\lbr(\C)$ is chosen to be the minimum of this optimization problem, and $\xr(\C)$ is chosen to be a minimizer. We note that computing each lower bound for $\alpha$BB is slow in comparison with the simpler  Lipschitz gradient method, as it requires solving a general box constrained convex optimization problem, which will typically requires several function evaluations for every lower bound computation. On the other hand, the bounds computed by $\alpha$BB are generally tighter. In fact, for non-degenerate problems, the bounds computed by $\alpha$BB for cubes in the vicinity of global minimizers are often exact, since in these cubes $f$ is strictly convex and so when they are small enough typically $f= \ell_\alpha $. We call this phenomenon \emph{eventual exactness} This in turn leads to the \emph{finite termination property}: for non-degenerate problems $\alpha$BB is able to find the exact solution in a finite number of steps (assuming that the solution to the convex optimization subproblems is computed exactly).

Algorithms with convergence order 3 are less common, but can be achieved for $C^2(\C_0)$ functions with Lipschitz Hessian, using minimization of the second order Taylor approximation of the function corrected according to the Hessian Lipschitz constant to ensure that a lower bound is obtained. For details on such a method see \cite{cartis2015branching,fowkes2013branch}.  

\section{Quasi BnB}\label{sec:qBnB}
Our main focus in this paper is introducing quasi-lower bounds
\begin{definition}[quasi-lower bound]
	We say that $\qlbr: \K \to [-\infty,\infty] $ is a quasi-lower  bound rule for minimizing $f$ over $\C_0$, if
	$$\qlbr(\C) \leq f_*(\C), \text{ for all } \C \in \K \text{ such that } f_*(\C)=f_*. $$
\end{definition}
We note that lower bound rules (Definition~\ref{def:lbr}) are necessarily quasi-lower bound rules, while quasi-lower bound rules are not necessarily lower bound rules. Nonetheless, in the context of BnB algorithms quasi-lower bounds can replace lower bounds without affecting the correctness of the algorithm. Recall that lower bound rules $\lbr$ are used to prove that a subcube $\C \subseteq \C_0$ does not contain a global minimizer, via inequalities of the form $\lbr(\C)>\ub $, where $\ub$ is an upper bound for $f_*$. Our simple but central observation is that if a similar inequality $\qlbr(\C)>\ub $ holds for a quasi-lower bound rule $\qlbr$, then we  also have a certificate that $f$ is not minimized in $\C$. This is because if $f$ were minimized in $\C$ then by definition of a quasi-lower bound rule we would have 
$$\qlbr(\C)\leq f_*(\C)=f_* \leq \ub .$$

We use the term \emph{quasi branch and bound algorithm (qBnB)} for the algorithm obtained from a BnB algorithm by replacing lower bound rules with quasi-lower bound rules. Thus, like BnB algorithms, qBnB algorithms are determined by a sampling rule and a quasi-lower bound rule, together with a  strategy for search the evolving list of subcubes visited by the algorithm. In Appendix~\ref{app:convergence} we describe a simple breadth first search (BFS) strategy and prove that qBnB algorithms converge to an $\epsilon$ optimal solution after a finite number of iterations.

We define the notions of convergence order and eventual exactness for qBnB algorithms in analogy to the definitions for BnB algorithms in Subsection~\ref{subsec:BnB}:
\begin{definition}\label{def:conv_order}
	Let $\xr$ and $\qlbr$ be sampling and quasi lower bound rules for minimizing a continuous function $f$ over a cube $C_0 \subseteq \RR^d $. 
	\begin{enumerate}
		\item We say that $(\xr,\qlbr) $ have convergence  order $\alpha$ for some $\alpha>0$, if there exists some $c>0$ such that for all $\C \in \K$ with radius $r$,
\begin{equation}\label{eq:quasi_order}
f(\xr(\C))-\qlbr(\C)\leq c\radius^{\alpha}.
\end{equation}
\item We say that $(\xr,\qlbr) $ are \emph{eventually exact} if there exists $\delta>0$ such that for all $\C \in \K$ contained in a ball of radius $\delta$ around a global minimizer, 
$$f(\xr(\C))-\qlbr(\C)=0 $$
\end{enumerate}
\end{definition}
The following sections are devoted to second order and third order qBnB algorithms, and their possible advantages over contemporary BnB algorithms.
\section{Second order Quasi BnB algorithms}\label{sec:second_order}
In this Section we suggest  quasi-BnB algorithms with convergence order $2$. In Subsection~\ref{sub:second_order} we describe a simple algorithm qBnB(2) with convergence order 2 for unconstrained optimization on a cube, and show this algorithm is provably tighter than the Lipschitz gradient algorithm . In Subsection~\ref{sub:constrained} we explain how qBnB(2) can be extended to constrained optimization on the cube, obtaining an algorithm we name \emph{constrained-qBnB(2)}. Our experiments (see Table~\ref{tab:constrained} and Section~\ref{sec:results}) show that the runtime of  constrained-qBnB(2) and qBnB(2) for unconstrained problems is similar, and that it can be two to seven times faster than the Lipschitz gradient algorithm. 

\subsection{Second order quasi-BnB}\label{sub:second_order}
In this subsection we consider unconstrained optimization problems, and assume that $f$ is  a $C^1(\C_0) $ function with Lipschitz gradients, and we are given a (possibly non-optimal) Lipschitz gradient constant $L_2\geq 0$. To define a quasi-lower bound, note that if $\C  \in \K$ is a cube which contains a minimizer $x_* \in \C_0$, then 
\begin{equation}\label{eq:fun_bound}
f(x)\leq f(x_*)+g^T(x_*)(x-x_*)+\frac{L_2}{2} \|x-x_*\|^2, \quad \forall x \in \C
\end{equation}
If $f\in C^2(\C_0)$ we can also take $L_2$ to be an upper bound for the maximal eigenvalue of $H(x)$ for $x \in \C $. Since $(f,\C_0) $ is unconstrained,  $g(x_*)=0$, and by choosing $x$  in \eqref{eq:fun_bound} to be the center $x_\C$ of $\C$ and denoting the radius of $\C$ by  $\radius$, we obtain a lower bound for $f$ in $\C$ via
\begin{equation} \label{eq:quasi2}
f(x_*)\geq f(x_\C)-\frac{L_2}{2}\radius^2
\end{equation}
Based on this observation, we define the qBnB(2) algorithm by the  the sampling and quasi lower bound rules 
\begin{equation}\label{eq:qBnB2rule}
\xr(\C)=x_\C \text{ and } \qlbr(\C)=f(\xr(\C))-\frac{L_2}{2}\radius^2 . \end{equation}
By \eqref{eq:quasi2} we see that $\qlbr$ is indeed a valid quasi-lower bound rule. Furthermore qBnB(2) has convergence order $2$ since
\begin{equation}\label{eq:uncertain2}
f(\xr(\C))-\qlbr(\C)=\frac{L_2}{2}\radius^2  
\end{equation}

One attractive attribute of qBnB(2) is its simplicity. Like classical Lipschitz optimization, but unlike second order BnB approaches,  computing the quasi-lower bound in this case only entails a single function evaluation at the center of the cube. This can be an important advantage for \emph{derivative free optimization} problems (e.g. \cite{audet2017derivative,larson2019derivative,conn2009introduction}), where the function minimized is differentiable, but its derivatives are not accessible. For such functions we are not aware of other algorithms which can achieve second order convergence. The simplicity of qBnB(2) also enables us to use it  for more complicated problems with additional structure where standard  second order methods might be difficult to adapt. In fact, qBnB(2) was first introduced in \cite{dym2019linearly} as a stepping stone towards solving optimization problems that while non-differentiable, are `conditionally Lipschitz differentiable'. For such functions the qBnB(2) framework can be successfully adapted to obtain a second order qBnB algorithm, while devising classical BnB algorithms for these problems with convergence order $>1$ seems to be a challenging task in general, and indeed competing methods for these problems typically have first order convergence \cite{yang2015go,pfeuffer2012discrete,hartley2009global}. This example is discussed in Appendix~\ref{app:conditional_second_order}.   

Another attractive attribute of qBnB(2) is that its bounds are tighter than the Lipschitz gradient  BnB algorithm, while the computational effort for computing the bounds for qBnB(2) is slightly smaller. We record this simple fact in the following proposition
\begin{proposition}\label{prop:better}
Let $(f,\C_0) $ be an unconstrained optimization problem, let $\lbr$  be the lower bound rules for the Lipschitz gradient algorithm (see~\eqref{eq:lip_grad_rule}) and $\qlbr$ the quasi-lower bound rule for qBnB(2) (see~\eqref{eq:qBnB2rule}). Then for every $\C \subseteq \C_0 $,
$$\lbr(\C) \leq \qlbr(\C) $$
\end{proposition} 
\begin{proof}
For every $\C \subseteq \C_0 $ with center $x_\C $ and radius $r$,
$$\lbr(\C)=\min_{x \in \C} f(x_\C)+g^T(x_\C)(x-x_\C)-\frac{L_2}{2} \radius^2\leq  f(x_\C)-\frac{L_2}{2} \radius^2=\qlbr(\C)  $$  
 \end{proof}

\subsection{Constrained quasi-BnB}\label{sub:constrained} 
A disadvantage of qBnB(2) in comparison with the Lipschitz gradient algorithm, is that the former is only valid for unconstrained problems over the cube, since it assumes the gradient at the minimizer vanishes. We now show that qBnB(2) can be adapted to the constrained scenario as well, leading to an algorithm we name constrained-qBnB(2). While constrained-qBnB(2)  is slightly less simple, it has the same attractive attributes as qBnB(2): It is a second order algorithm, and only requires a bound on the variation of the gradient, but does not need to compute the gradient itself in any step of the algorithm. We note that the method we suggest for dealing with box constraints below can probably be extended to general convex polyhedrons.   

In Constrained-qBnB(2) we assume $f\in \C^1(\C_0)$ with a Lipschitz gradient constant $L_2$, but do not assume as in qBnB(2) that $(f,\C_0) $ is unconstrained. Out strategy is to choose a single point $\xr(\C)$ for each cube, with the property that for any other point $x\in\C$, there exists some $\epsilon>0$ such that  
\begin{equation}\label{eq:rel_int}
tx+(1-t)\xr(\C) \in \C, \quad \forall t\in [-\epsilon,1] 
\end{equation} 
As a result any minimizer in $\C$ will be an unconstrained minimizer with respect to the line between $\xr(\C) $ and $x_*$, which will be enough for deriving a quasi-lower bound analogous to the one used in qBnB(2). We will soon explain this in detail, but we will first describe our method for choosing $\xr(\C)$.

\begin{wrapfigure}[4]{r}{0.3\columnwidth}
	\vspace{-1.5 em}
	\centering
	\includegraphics[width=0.3\columnwidth]{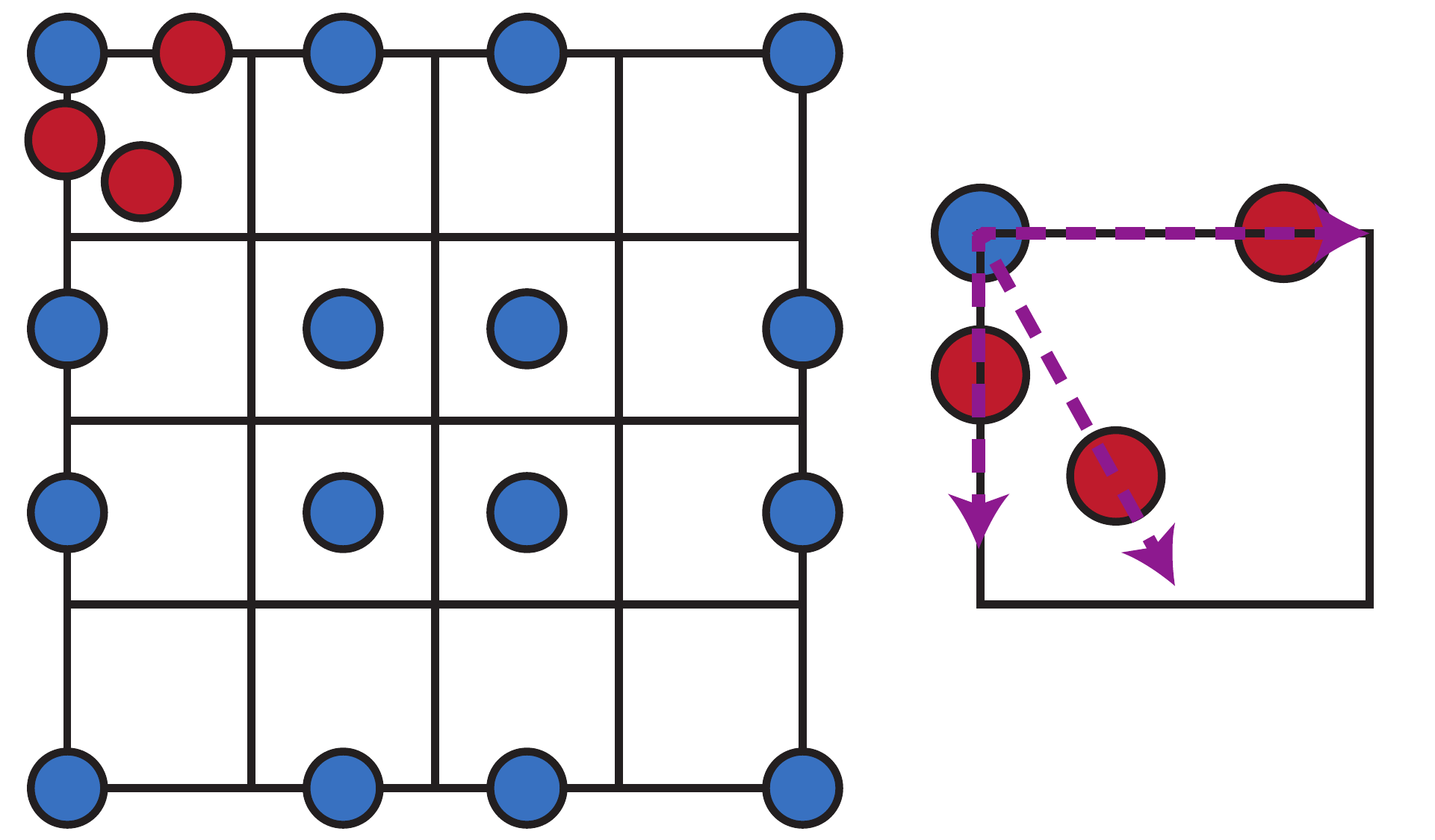}
	\vspace{-2 em}
	\caption{sampling scheme}
	\label{fig:constrained}
\end{wrapfigure}
We write $\C_0=\prod_{i=1}^d [a_i,b_i]$, and let $\C=\prod_{i=1}^d [c_i,d_i]$ be some subcube. We assume that $d_i-c_i<b_i-a_i, \forall i, i=1,\ldots d $. For large cubes which do not satisfy this condition we just set $\qlbr(\C)=-\infty, \xr(\C)=x_\C$. For cubes which do satisfy this condition we set  
\begin{equation}\label{eq:xr}
\xr_i(\C)=\threepartdef{c_i}{c_i=a_i}{d_i}{d_i=b_i}{\frac{c_i+d_i}{2}}
\end{equation}  
It can be verified that \eqref{eq:rel_int} is satisfied for every $x \in \C$ and a sufficiently small $\epsilon$. This sampling scheme is illustrated in Figure~\ref{fig:constrained}, where the blue dots stand for the points sampled in each cube $\C$, and the red points illustrate the fact that a line between the blue point and red point can be extended beyond the red point.

Now let $\C$ be a cube for which $\qlbr(\C)>-\infty $, and assume that $x_* \in \C $ is a minimizer, then the inner product of $g(x_*)$ with the vector $\xr(\C)-x_*$ vanishes, due to \eqref{eq:rel_int}, and so 
 \begin{align}\label{eq:constrained_derivation}
 f(\xr(\C)) &\leq f(x_*)+g^T(x_*)(\xr(\C)-x_*)+\frac{L_2}{2} \|\xr(\C)-x_*\|^2= f(x_*)+\frac{L_2}{2} \|\xr(\C)-x_*\|^2 \nonumber\\
 &\leq f(x_*)+\frac{L_2}{2} \max_{x \in \C}\|\xr(\C)-x\|^2\\
 &= f(x_*)+\frac{L_2}{2}\sum_{i=1}^d \max \{(\xr_i(\C)-c_i)^2, (\xr_i(\C)-d_i)^2 \} \nonumber
   \end{align}
   Accordingly we define constrained-qBnB(2) algorithm via the sampling rule \eqref{eq:xr} and the quasi-lower bound rule 
   \begin{equation*}
 \qlbr(\C)=f(\xr(\C))-\frac{L_2}{2}\sum_{i=1}^d \max \{(\xr_i(\C)-c_i)^2, (\xr_i(\C)-d_i)^2 \}.
   \end{equation*}
The derivation presented above shows that $\qlbr$ is a valid quasi-lower bound since $\qlbr(\C) $ is a lower bound when $\C$ contains a minimizer due to \eqref{eq:constrained_derivation}. Moreover the pair $(\xr,\qlbr) $ has convergence order $2$, and for $\C$ which do not intersect $\C_0$ the sampling and quasi-lower bound rules are in fact identical to those used in qBnB(2).

\section{An eventually exact third order qBnB algorithm }\label{sec:third_order}
In Section~\ref{sec:second_order} we discussed second order quasi-BnB algorithms. In this section we introduce a third order qBnB algorithm for unconstrained optimization over a cube, which we name qBnB(3). Our assumptions in this section are that $f\in C^2(\C_0)$, and $H(x)$ is Lipschitz in $\C_0$ with a known Lipschitz constant $L_3\geq 0$.  We further assume that minimizing $f$ over $\C_0$ is an unconstrained optimization problem. 

 Besides having third order convergence, qBnB(3) is also eventually exact (Definition~\ref{def:conv_order}), providing that  the unconstrained optimization problem is non-degenerate. In comparison with the $\alpha$BB algorithm discussed in Subsection~\ref{subsec:BnB}, which is eventually exact as well, the main advantage of qBnB(3) is that computing the quasi lower bound for each cube only requires solving strictly convex unconstrained optimization problems via Newton iterations in the regime of rapid quadratic convergence, and so requires a very small number of Newton iterations. In contrast, $\alpha$BB solves a box-constrained convex optimization problem in each cube. 
 
 The basic idea behind qBnB(3) is as follows: If $\C \subseteq \C_0$ is a cube which contains a minimizer $x_*$, then  $H(x_*)\succeq 0$. In this case, by adding a strictly convex quadratic regularizer of magnitude $O(r^3) $ (where $r$ is the radius of the cube) we obtain a new function $\hat f$ which is strictly convex in a neighborhood of $\C$, and for which we can prove that Newton iterations initialized from the center of $\C$ converge quadratically. Accordingly, for every cube (even those which do not contain a minimizer) we first verify that the Hessian of $f$ in the center of the cube is at least close to convex, and then run Newton iterations and check whether they achieve quadratic convergence. When quadratic convergence fails we obtain a certificate that $\C$ does not contain a minimizer, and when quadratic convergence occurs the limit point can be used to obtain third order sampling and quasi-lower bound rules.
 
 The following proposition is the first stage towards making these ideas precise.
\begin{proposition}\label{prop:reg}
Let $(f,\C_0) $ be an unconstrained optimization problem.  Let $U_0$ be an open set containing $\C_0$, such that \eqref{eq:unconstrained} holds, and  $f$ is in $C^2(U_0) $. Let $\C \in \K$ be a cube  with radius of length $r_0=r$ and center $x_0=x_\C$, and assume that $\bar B_{2r}(x_0) $ is contained in $U_0$. Finally assume  that $L_3\geq 0$ is a Hessian Lipschitz constant for $H(x)$ in  $\bar B_{2r}(x_0) $. Let $\hat f$ denote the auxiliary function
\begin{equation}\label{eq:hat_f}\hat f(x)=f(x)+\frac{\bar \lambda}{2}\|x-x_0\|^2 \text{ where } \bar \lambda=\max\{0,5L_3r-\lambdamin(x_\C)  \} .\end{equation}
Furthermore define
\begin{equation}\label{eq:rk}
m=3L_3r, r_0=r \text{ and } r_{k+1}=\frac{3L_3}{2m}r_k^2 , \quad \forall k\geq 0 \end{equation}
Let $x_k$ denote the Newton iterations for minimizing the auxiliary function,  initialized from $x_0$. If there exists a minimizer $x_*$ for $f$ in $\C$ then 
	\begin{align}
	\lambdamin(x_0)&\geq -L_3r \label{subeq:1}\\
	\|x_{k+1}-x_k\| &\leq r_k+r_{k+1} \text{ and } \|x_{k}-x_0\|\leq r_{k}+r_0 , \forall k\geq 0 . \label{subeq:2}
	\end{align}
Additionally $x_k$ converge to a minimizer $\hat x_* $ of $\hat f$, and $f_*$ is bounded by
	\begin{equation}\label{eq:L3bound}
	f(\hat x_*) \geq f_*\geq \hat f(\hat x_*)-\frac{\bar \lambda}{2}r^2
	\end{equation} 

	\end{proposition}
\begin{proof} 
	Assume  $\C$ contains a minimizer $x_*$ for $f$. Then since $x_0-x_* \leq r $ and  $\lambdamin(x_*)\geq 0$ we obtain \eqref{subeq:1}. 
	
	We use $\hat H(x)$ and $\hat g(x)$ to denote the Hessian and gradient of the new function $\hat f$ at a point $x$. If  $x \in B_{2r}(x_0)$ then  $\hat H(x) \succeq m I_d $ where $m= 3L_3r $ because
	\begin{equation}\label{eq:psd}
	\hat H(x)=H(x)+\bar \lambda I_d\succeq H(x_0)+(\bar \lambda-2L_3r)I_d \succeq 3L_3rI_d+[H(x_0)-\lambdamin(x_0)I_d] \succeq 3L_3rI_d
	\end{equation}
		For any $x \in U_0$ such that $\|x-x_0\|\geq r $, since $x_* $ minimizes $f$ and $\|x_*-x_0\| \leq r$, necessary $\hat f(x_*)\leq \hat f(x) $, and thus $\hat f$ must be minimized in $\bar B_r(x_0) $. Since in \eqref{eq:psd} we saw that $\hat f$ is strictly convex in $B_{2r}(x_0) $ if follows that $\hat f  $ has a unique global minimizer in $\bar B_r(x_0) $.

 We now consider the process of  minimizing $\hat f$ using Newton iterations, initialized from $x_0$: 
\begin{equation} \label{eq:Newton}
x_{k+1}=x_k-\hat H(x_k)^{-1} \hat g(x_k) .
\end{equation} 
 The distance of $\hat x_* $ from $x_0$ is bounded by $r_0=r$. For all $k>0$ we can then iteratively bound the distance of $x_{k+1}$ from $\hat x_*$  by $r_{k+1}$ because 
\begin{align*}
\|x_{k+1}-\hat x_*\|\stackrel{\eqref{eq:Newton}}{=} & \|x_k-\hat x_*- \hat H(x_k)^{-1}\hat g(x_k) \|=\|\hat H(x_k)^{-1}\left[\hat H(x_k)\left(x_k-\hat x_*\right)- \hat g(x_k) \right] \|\\
&\stackrel{\eqref{eq:psd}}{\leq} \frac{1}{m} \|\hat H(x_k)\left(x_k-\hat x_*\right)- \hat g(x_k)\| \\
&\leq \frac{1}{m}\|\left(\hat H(x_k)-\hat H(\hat x_*)\right)\left(x_k-\hat x_*\right)\|+\frac{1}{m}\|\left(\hat g(\hat x_*)+\hat H(\hat x_*)(x_k-\hat x_*) \right) -\hat g(x_k) \|\\
&\leq  \frac{3L_3}{2m}\|x_k-\hat x_*\|^2\leq   \frac{3L_3}{2m}r_k^2 = r_{k+1}
\end{align*}
We now have \eqref{subeq:2} directly from the triangle inequality. 

We note that $r_k$ converges quadratically to zero, since the quantity $\frac{3L_3}{m}r_k $ is equal to $1$ for $k=0$, and decreases quadratically:
$$\frac{3L_3}{m}r_{k+1} = \frac{1}{2} (\frac{3L_3}{m}r_{k} )^2  .$$
It follows that $x_k$ converges to a point in $\bar B_r(x_0) $, and since $\hat f$ is strictly convex in $B_{2r}(x_0) $ this point must be $\hat x_*$. 
Now the left hand side of \eqref{eq:L3bound} follows immediately from the definition of $f_*$, and the right hand side follows from
$$\hat f(\hat x_*)\leq \hat f(x_*)=f(x_*)+\frac{\bar \lambda}{2}\|x_*-x_0\|^2 \leq f(x_*)+\frac{\bar \lambda}{2}r^2 .$$
\end{proof}

\paragraph{Algorithm}
Based on Proposition~\ref{prop:reg} we suggest the following qBnB algorithm, which we name qBnB(3): Assume we are  given a cube $\C \subseteq \C_0$ with center $x_0=x_\C$ and radius $r$.  Assume that $\bar B_{2r}(x_0) \subseteq U_0 $ where $U_0$ is an open set described in Proposition~\ref{prop:reg} (otherwise set $\qlbr(\C)=-\infty $), and we are given a Lipschitz  bound $L_3$ for the Hessian function $H(x), x \in B_{2r}(x_0) $. We define the quasi-lower bound and sampling rules for $\C$ as follows:
\begin{enumerate}
	\item Compute $H(x_0) $ and check that \eqref{subeq:1} holds. If it doesn't, we know that $\C$ does not  contain a minimizer and so we set $\qlbr(\C)=\infty $ and $\xr(\C)=x_0$.
	\item If  \eqref{subeq:1} does hold, we consider the Newton iterations $x_k$ for minimizing $\hat f$, as defined in \eqref{eq:Newton}. For each $x_k$ we check whether \eqref{subeq:2} holds, where $r_k$ are defined as in \eqref{eq:rk}. If for some $k$ the condition isn't satisfied we know $\C$ does not contain a minimizer and so we set $\qlbr(\C)=\infty $ and $\xr(\C)=x_0$. Otherwise \eqref{eq:hat_f} implies that $x_k$ is a Cauchy sequence, and so it has a limit which we denote by $\hat x_*$. We then  set 
	\begin{equation*}
	\xr(\C)=\hat x_* \text{ and } \qlbr(\C)=\hat f(\hat x_*)-\frac{\bar \lambda}{2}r^2 . 
	\end{equation*}
\end{enumerate} 

The following theorem shows that the qBnB(3) algorithm it indeed a valid qBnB algorithm, with third order convergence and, in the case of non-degenerate problems, eventual exactness .
	
	\begin{theorem}\label{thm:qBnB3}
		Let $(f,\C_0) $ be an unconstrained optimization problem. Assume that $U_0$ is an open set containing $\C_0$, such that   $f \in C^2(U_0) $ and \eqref{eq:unconstrained} holds. Let and $L_3$ be a Lipschitz constant for the Hessian function $U_0 \ni x \mapsto H(x)$. Let $\qlbr,\xr$ be the quasi-lower bound and sampling rules for the qBnB(3) algorithm. Then
		\begin{enumerate}
	\item $\qlbr$ is a valid quasi-lower bound rule.
	\item \textbf{Third order convergence.} For any cube $\C \subseteq \C_0$ we have
	\begin{equation}\label{eq:3correct}
	 f(\xr(\C))-\qlb(\C)\leq  3L_3r^3
	\end{equation}
	\item \textbf{Eventually exact.} If $f$ has a finite number of minimizers in $\C$ and the Hessian at each minimizer is strictly positive definite, then $(\xr,\qlbr) $ is eventually exact.
	\end{enumerate}
	\end{theorem}
\begin{proof}
	\begin{enumerate}
\item If $\C\in \K$ is a cube which contains a minimizer, by Proposition~\ref{prop:reg} we have that $\qlbr(\C)\leq f_*$ and so the quasi lower bound is valid.
\item For every $\C \in \K$, if $\qlb(\C)=\infty$ then \eqref{eq:3correct} holds. Otherwise condition \eqref{subeq:1} holds and so
$$ \bar \lambda=\max\{0,5L_3r-\lambdamin(x_0)\}\stackrel{\eqref{subeq:1}}{\leq} \max\{0,6L_3r\}=6L_3r,$$
which implies that 
\begin{equation}\label{eq:3error}
f(\xr(\C))-\qlbr(\C)=f(\hat x_*)-(\hat f(\hat x_*)-\frac{\bar \lambda}{2}r^2)=\frac{\bar \lambda}{2}(r^2-\|\hat x_*- x_0\|^2)\leq  3L_3r^3.\end{equation}
\item If $f$ has a finite number of minimizers in $\C$ and the Hessian at each minimizer is strictly positive definite, then for $\delta>0$ small enough, for any minimizer $x_*$ and $x_0 \in \C_0$ with $\|x_0-x_*\|\leq \delta $ we have that $\lambdamin(x_0)\geq 5L_3 \delta  $. Thus if $
\C_r(x_0) \subseteq B_\delta(x_*) $ we have that $r \leq  \delta $ and so 
$$\bar \lambda=\max\{0,5L_3r-\lambdamin(x_0)  \} =  0$$
which implies that $ f(\xr(\C))-\qlb(\C)=0$ due to the second equality from the left in \eqref{eq:3error}. 
\end{enumerate}
\end{proof}

To summarize our discussion so far, we have defined qBnB(3) and showed that is it indeed a well defined qBnB algorithm, that it enjoys third order convergence, and for non-degenerate problems, eventual exactness. We conclude this section with some practical notes on implementation of the qBnB(3) algorithm.
\subsection{Implementation details}
\paragraph{Stopping criterion for Newton's algorithm} In practice we cannot of course find $\hat x_*$ with zero error. However since we are in the regime of  rapid convergence of Newton's  method, we can achieve very accurate approximations extremely quickly. In practice we allow the Newton algorithm an error of $\epsNewton<0.01 \epsilon $, where $\epsilon$ is the requested accuracy for the global optimization process. To ensure the error of Newton's algorithm is smaller than $\epsNewton$, we use the fact that the maximal eigenvalue of $\hat f$ in $\bar B_{2r}(x_0) $ is bounded by 
$$M=\lambdamax(x_0)+\bar \lambda+2L_3r . $$
By considering Taylor expansion of $\hat f$ around its minimizer $\hat x_*$ 
$$\hat f (\hat x_*) \geq \hat f(x_k)- M/2 \cdot r_k^2 .$$
Accordingly we stop the Newton iterations when
$$\epsNewton\geq M/2 \cdot r_k^2. $$ 
 Denoting the iteration in which the Newton algorithm was stopped  by $K$, the quasi lower bound and sampling rules are  computed as
$$\xr(\C)=x_K \text{ and } \qlbr(\C)= \hat f(x_K)- \frac{\bar \lambda}{2}r^2-\epsNewton.$$

\paragraph{Combining qBnB(2) and qBnB(3)} In practice we suggest to used a method which combines qBnB(2) and qBnB(3) methods, for two reasons:
(i) qBnB(3) requires a bound $L_3$ on the Hessian Lipschitz bound for  all points in $\bar B_{2r}(x_0) $ which strictly contains the cube $\C$ and may not be contained in $\C_0$. (ii) In practice (see Figure~\ref{fig:rastriginGraphs}) we find that the bounds computed by qBnB(3) are more efficient for small cubes, but less efficient for large cubes. Accordingly, we suggest to use  third order bounds only for cubes $\C$ with radius $r$ for which (i) the closed ball $\bar B_{2r}(x_\C) $ is contained in $\C_0$ and  (ii) the radius $r$ is small enough so that the qBnB(3) bounds are expected to be more accurate than the qBnB(2) bounds. This occurs when the uncertainty  estimates for qBnB(3) from \eqref{eq:3correct} are smaller than the corresponding uncertainty  estimates for qBnB(2) in \eqref{eq:uncertain2}, that is
$$\frac{L_2}{2}r^2 \geq 3L_3r^3.$$ 
We name this combined quasi-BnB algorithm qBnB(2+3).

\section{Experiments}\label{sec:results}
In this section we describe several experiments we conducted to compare the various BnB and qBnB algorithms discussed in the paper. Before describing the experiments themselves, we describe our method for producing valid Lipschitz bounds $L_s, s=1,2,3$. 

\subsection{Computing Lipschitz bounds}\label{sub:Lip}
We use a very simple protocol to compute  Lipschitz bounds $L_s, s=1,2,3$. For alternative methods for computing Lipschitz constants see e.g., \cite{cartis2015branching}. Firstly, for simplicity we compute a single Lipschitz bound which is valid over the whole cube $\C_0$. In general it is possible to compute a Lipschitz bound per cube, which could lead to better results.

 To compute a Lipschitz bound $L_1$ over $\C_0$, for a function $f \in \C^1(\C_0) $, we bound the function $\|\nabla f(x)\|$ on $\C_0$ using interval arithmetic (see e.g., \cite{dawood2011theories}). We use a similar strategy for $L_2$ and $L_3$ as well. To compute  $L_2$ over $\C_0$, for a function $f \in \C^2(\C_0) $, we use interval arithmetic to bound the Frobenius norm $\|H(x)\|_F $ over $\C_0$, which in turn upper bounds the operator norm of $H(x)$. For $s=3 $ we  consider the $3 $ dimensional tensor $T(x)$ with $d^3 $ entries which is composed of all derivatives of $f$ of order $3$. In \cite{fowkes2013branch} it is shown that valid bounds $L_3$ can be computing by bounding the Frobenius norm of the tensors $T(x)$, which is  defined as 
\begin{equation}\label{eq:tensor}
\|T(x)\|_F=\left[ \sum_{\alpha \in \{1,\ldots,d\}^3} [\partial^\alpha f(x)]^2 \right]^{1/2}  .
\end{equation}

\subsection{Rastrigin}
In our first experiment, we compare the performance of the three qBnB algorithms we suggested, qBnB(2), qBnB(3) and qBnB(2+3), with two popular BnB algorithms described in Subsection~\ref{subsec:BnB}: The canonical Lipschitz algorithm and the $\alpha$BB algorithm. The algorithms were all implemented in Matlab using the breadth first search technique described in Algorithm~\ref{alg}. The $\alpha $BB algorithm requires a bound on the minimal eigenvalue of the Hessian for all points in a given cube $\C$. We use the bound   $\lambdamin(x_\C)-L_3 r $, where $r$ is the radius of the cube. The Convex optimization sub-problems in $\alpha$BB were solved using Matlab's fmincon. The gradient of the functions was supplied and the sqp optimization algorithm was used as we found it faster than the other fmincon algorithms for the problems at hand. 

In this experiment we consider the problem of optimizing the function $f:\RR^d \to \RR $ defined for $\alpha \in \RR^d $  and $\delta,\theta \in \RR$ by 
 \begin{equation} \label{eq:Rastrigin}
  f(x)=\sum_{i=1}^d \alpha_i(1-\cos(\theta x_i))+\delta\|x\|^2 
 \end{equation}
over the parameter domain $[-a,a]^d $. 
For parameter choices of  
\begin{equation}\label{eq:rastrigin_params}
\theta=2\pi, \, a=5.12 , \,  \delta=1, \,  \alpha_i=10, \quad  \forall i\in \{1,\ldots,d\} 
\end{equation}
this function is the well-known Rastrigin function. When $\delta>0$ The function $f$ has a unique minimizer $x_*=0$. In this example we do not use interval arithmetic, but compute the bounds analytically, as it is straightforward to bound the Frobenius norm of the gradient, Hessian, or the tensor in   \eqref{eq:tensor} on all of $[-a,a]^d$ by
\begin{align*}
L_1=& \|\alpha\|_2|\theta|  +2|\delta|\sqrt{d}|a| \\
L_2=& \|\alpha\|_2|\theta|^2+2|\delta| \\
L_3=&\|\alpha\|_2 |\theta|^3 
\end{align*}

Figure~\ref{fig:rastriginGraphs} compares the five algorithms when applied to the Rastrigin function with the standard parameter values \eqref{eq:rastrigin_params} in the 2-dimensional case ($d=2$), with requested accuracy of $\epsilon=10^{-8}$ and a time limit of ten minutes. 

Figure~\ref{fig:rastriginGraphs}(a) plots for each algorithm the cumulative number of cubes visited  up to a given depth level of the search tree. We see that, as expected by the theoretical analysis of the clustering problem, the number of cubes visited by the Lipschitz algorithm increases dramatically as the search tree becomes deeper (corresponding to higher accuracy). In contrast, the remaining algorithms spend most of their iterations in the first stages, where the cubes are large and  the computed bounds are not effective. We note also that qBnB(3) takes longer to escape the initial stage of ineffective bounding, but qBnB(2+3) is comparable to qBnB(2) in the first stages. 

Figure~\ref{fig:rastriginGraphs}(b) plots for each algorithm the cumulative time  up to a given depth level of the search tree. As expected from Figure~\ref{fig:rastriginGraphs}(a), the Lipschitz algorithm is much slower than its competitors, and in fact timed out after ten minutes without achieving the requested accuracy of $\epsilon=10^{-8}$. We note that while the number of cubes visited by $\alpha$BB was comparable to the number of cubes visited by the qBnB algorithms, its timing is significantly slower due to the relatively high cost of computing the lower bound at each cube via solving a box constrained convex optimization problem.  

Figure~\ref{fig:rastriginGraphs}(c) illustrates the finite termination property enjoyed by the algorithms $\alpha$BB, qBnB(2) and qBnB(2+3), via the sudden drop to zero of the error computed by the algorithm.  Figure~\ref{fig:rastriginGraphs}(d) shows the timing of the various algorithms as a function of the dimension $d$ of the problem. As expected, all algorithms become significantly more expensive as the dimension is increased. In this figure we do not include the Lipschitz algorithm since it was not able to attain the required accuracy within the time limit for any one of the values of $d$ we checked.

\begin{figure}[t]
	\includegraphics[width=\textwidth]{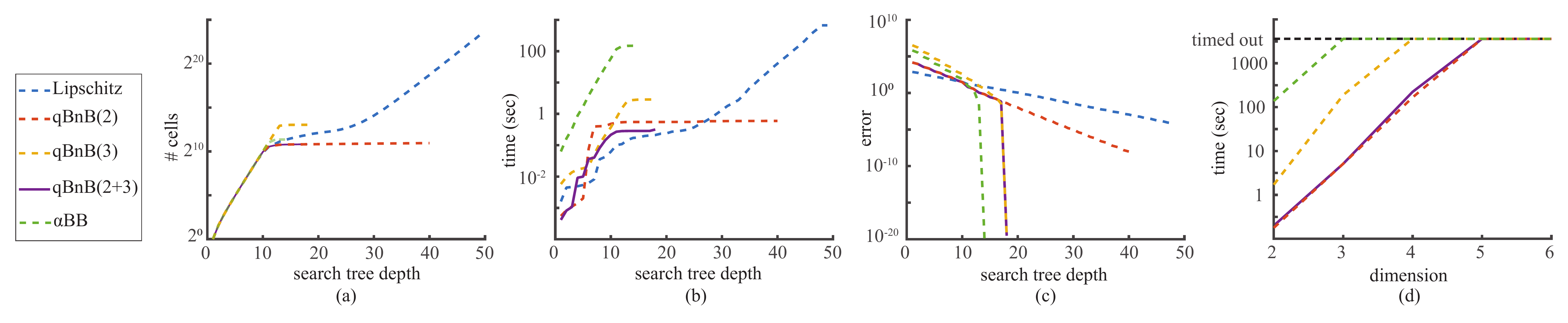}
	\caption{comparison of BnB and qBnB algorithms for the task of optimizing the Rastrigin function, in term  of (a) the cumulative number of cubes each algorithm visited in each depth level of the search tree (b) cumulative time vs. depth (c) the error $\ub-\lb$ computed by the algorithm at each depth level. In (d) we show examine the dependence of the runtime of the algorithms on the dimension $d$ of the problem. }
	\label{fig:rastriginGraphs}
\end{figure}

\subsection{Dixon-Szego}\label{sub:dixon}
Our results for the first experiment suggest that qBnB(2) and qBnB(2+3) are comparable in terms of timing, and outperform the remaining algorithms. In our second experiment we compared qBnB(2), qBnB(2+3) and $\alpha$BB on the nine Dixon-Szego test functions \cite{dixon1978global}. We compute the  bounds $L_2,L_3$  using interval arithmetic in Mathematica as described in Subsection~\ref{sub:Lip}. Table~\ref{table:dixon} shows the number of seconds the algorithms required to obtain global optimality up to a $10^{-8}$ error, per each optimization problem.  The algorithms were stopped if their runtime exceeded 1 hour, in which case the accuracy obtained after one hour is shown (denoted by `acc' in the table). The table shows that qBnB(2+3) and qBnB(2) both considerably outperformed $\alpha$BB in terms of efficiency. The two qBnB algorithms performed comparably on most of the problems. For the Goldstein-Price and Hartman3 functions qBnB(2+3) was significantly faster.

Table~\ref{table:dixon} also shows the dimensions of the different problems and the computed values of the bounds $L_i, i=2,3$.

 \begin{table}
 	\begin{center}
 		\begin{tabular*}{\textwidth}{l @{\extracolsep{\fill}} lllllll}
 			\hline
 			Problem     	&	d   & qBnB(2)   &qBnB(2+3)&	$\alpha$BB  & $L_2$ & $L_3$ \\
 			\hline
 			Branin	    	&   2   & 0.1 sec   &	0.1 sec  &	8 sec & 41 &	14.1 \\
 			Camelback   	&	2	& 0.3 sec	& 0.6 sec	 &12 sec  & 949  &	1100\\
 			Goldstein-Price	&   2	& 2608 sec	& 82 sec	 &2659 sec & $5.4 \times 10^8$ &	$8.7 \times 10^8$ \\
 			Shubert		    &   2   & 1 sec     &	 0.8 sec & 415 sec & $1.0\times 10^4$ &	$7.2 \times 10^4$ \\
 			Hartman3	    &   3	& 172 sec	& 118 sec	 & 1   (acc)& $2.2\times 10^4$ &	$1.2\times 10^6$\\
 			Shekel5	        &   4	& 0.1 (acc) &	0.1 (acc) &	1  (acc)& $4.0\times 10^3$ &	$2.4\times 10^5$ \\
 			Shekel7		    &   4   & 0.1 (acc) &	0.1 (acc)&	1  (acc)&$4.0\times 10^3$ &	$2.4\times 10^5 $\\
 			Shekel10	    &   4   &	0.1 (acc) &	0.9 (acc) &	6  (acc)&$4.0\times 10^3$ &	$2.4\times 10^5 $ \\
 			Hartman6	    & 6     &	0.03  (acc)& 	0.03 (acc) &	0.5 (acc)&$ 7.6 \times 10^3	$& $ 2.5 \times 10^5$  \\
 			\hline
 			
 		\end{tabular*}
 	\end{center}
 	\caption{Comparison of qBnB(2), qBnB(2+3), and $\alpha$BB on the Dixon-Szego test functions. The table shows the number of seconds the algorithms required to obtain global optimality up to a $10^{-8}$ error. The algorithms were stopped if their runtime exceeded 1 hour, in which case the accuracy obtained after one hour is shown. }
 	\label{table:dixon}
 \end{table}

\paragraph{Constrained optimization} In our final experiment we compared three algorithms: qBnB(2), constrained-qBnB(2) and the Lipschitz gradient algorithm described in Subsection~\ref{subsec:BnB}. We compared these algorithms on both constrained and unconstrained optimization problem, and the results are shown in Table~\ref{tab:constrained}.

We generated ten random unconstrained Rastrigin-like problems by setting $f$ as in \eqref{eq:Rastrigin}, where the  coefficient vectors $\alpha \in [0,1]^3 $ were generated randomly with uniform distribution, and setting $\delta, \theta,a$ as in \eqref{eq:rastrigin_params}, so that the unique minimizer of $f$ is zero, which is in the interior of $\C_0$ in this case. All three algorithms returned the correct solution, and the Lipschitz gradient algorithm was two times slower than the other two algorithms whose timing was similar, as shown in the right hand side of Table~\ref{tab:constrained} (timing is averaged over the ten experiments). This example indicates that there is not much to lose in using constrained-qBnB(2) even for problems where the minimizer is attained inside the cube.

We next generated ten random constrained problems in the same way as before, but changed the value of $\delta$ to $-1$ so that the minimum tended to be obtained on the boundary of the cube. As expected qBnB(2) did not find a solution within the required accuracy of $\epsilon=10^{-8}$. Constrained-qBnB(2) was on average more than seven times faster than the gradient-Lipschitz algorithm, and the difference in function value between the solutions obtained by the two methods was smaller than the error tolerance. These results are shown in the left hand side of Table~\ref{tab:constrained}. The error column represents the average deviation of a solver from the best solution attained per-problem.

\begin{table}
	\begin{center}
		\begin{tabular}{ccccc}
			\hline
			&	iterations     & error    & iterations & error \\
			& (constrained)& (constrained)& (unconstrained)& (unconstrained)\\
			\hline
			gradient-Lipschitz & $9.9 \times 10^3$ & $8.6 \times 10^{-9}$ & $9.9 \times 10^3$ & $0$ \\
			qBnB(2) & $2.5 \times 10^3$  & $2.5 \times 10^{-4}$ & $5.1 \times 10^3$ &$ 0$ \\
			constrained-qBnB(2) & $1.3 \times 10^3 $ & $0$ & $5.1 \times 10^3$ & $0$\\
			\hline
			\end{tabular}
	\end{center}
	\caption{Comparison of the  gradient-Lipschitz BnB algorithm with qBnB(2) and constrained-qBnB(2) for Rastrigin-like problems. For unconstrained problems qBnB(2) and constrained-qBnB(2) performed similarly, and were more effecient than the gradient-Lipschitz algorithm. For constrained optimization, as expected, qBnB(2) did not attain a solution within the required accuracy of $10^{-8}$. Constrained-qBnB(2) found the global minimizer faster than the gradient-Lipschitz algorithm. }
	\label{tab:constrained}
\end{table}

To conclude this section, the experiments  we conducted found that qBnB(2) and qBnB(2+3) outperform the remaining BnB and qBnB algorithms discussed in this paper. We also saw that while the timing of these two algorithms is often comparable, for certain problems qBnB(2+3) is significantly more efficient. Finally, we found that constrained-qBnB(2) is comparable to qBnB(2) when applied to unconstrained problems.

 The code used for running the experiments present in this section can be found in \cite{code}.  
 We note that the timing experiments presented here were all implemented in Matlab, and were not rigorously optimized for time efficiency. We expect all algorithms would equally benefit from implementation in compiled based programming languages such as C++.

\section{Conclusions and future work}
In this paper we introduced the notion of quasi-BnB in the context of general continuous optimization, and suggested several qBnB algorithms: The qBnB(2) algorithm is a very simple algorithm with second order convergence. In comparison with BnB algorithms, a major advantage is that is does not require computation of derivatives and so can provide a second order algorithm for derivative free optimization. Moreover, qBnB(2) is provably more efficient than the Lipschitz gradient algorithm. We next suggested how to generalize this algorithm to box constrained optimization, obtained the algorithm constrained-qBnB(2). Finally, we introduced qBnB(3) which is an algorithm with third order convergence and eventual exactness. Our experiments showed that qBnB(2+3), which combines qBnB(2) and qBnB(3), outperforms other competing BnB and qBnB algorithms we discussed in this paper. 

An interesting problem which we have not addressed in this paper is how to extend qBnB(3) to constrained optimization over a cube, and how to handle constraints which are more complicated than the box constraints discussed here.

\bibliographystyle{unsrt}
\bibliography{bib_smooth_quasi}

\appendix
\section{Convergence of qBnB}\label{app:convergence}
In this appendix we give a formal proof for the global convergence of qBnB algorithms. A qBnB algorithm is determined by a quasi-lower bound rule, a sampling rule, and a strategy for traversing  the  search tree. In our analysis here we will use  the simple breadth first search (BFS) algorithm described in Algorithm~\ref{alg}, which was the one we used for the experiments in Section~\ref{sec:results}.
\begin{algorithm}[H]
	\SetKwInOut{Input}{input}\SetKwInOut{Output}{output}
	\SetAlgoLined
	\Input{Required accuracy $\epsilon$, function $f$ and cube $\C_0$} 
	\Output{$\epsilon$-optimal solution $\xbest$}
	$g \leftarrow 0$ \;
	$q_0\leftarrow \qlbr(\C_0)$, $\xbest \leftarrow\xr(\C_0) $\;
	$\lb\leftarrow q_0 $, $\ub\leftarrow f(\xbest) $\;
	Put $(\C_0,q_0) $ into the list $L_g$ \;
	\While{$\ub-\lb>\epsilon$}{
		
		Initialize an empty list $L_{g+1}$\;
		\For{$(\C,q) \in L_g$}{
			\If{$q\leq \ub$ (cube cannot be eliminated) 	 \label{ln:eliminate}}
			{Subdivide $\C$ into two cubes $C_1,\C_2 $ along the longest edge of $\C$\;
				\For{ $i=1,2$}{  $q_i\leftarrow \qlbr(\C_i)$\;
					$ x_i\leftarrow \xr(\C_i) $ \;
					Add $(\C_i,q_i) $ to $L_{g+1}$\;
					\If{$f(x_i)<\ub$}{$\ub \leftarrow f(x_i)$ \;
						$\xbest \leftarrow x_i $\;}
				}
			}
		}
		$ g \leftarrow g+1 $ \label{g_update}\;
		$\lb=\min \{q| \, (\C,q) \in L_g  \} $ \label{lb:def} \;
	}
	\caption{Breadth first search qBnB algorithm} 
	\label{alg}
\end{algorithm}

The following Theorem shows that qBnB algorithms exhibit convergence to a global minimizer, providing that the difference between the upper bound and quasi lower bound computed per cube goes to zero with the diameter of the cube:
\begin{theorem}\label{thm:convergence}
	Let $\C_0$ be a cube in $\RR^d$, and $f $ a continuous function on $\C_0$. Let $\xr$ and $\qlbr$ be sampling and quasi-lower bounds rules. Fix $\epsilon>0$. Then
	\begin{enumerate}
		\item $\lb$ defined in Line~\ref{lb:def} is a lower bound for $f_*$.
		\item If there exists a continuous function $\psi:\RR_{\geq 0} \to \RR_{\geq 0} $ with $\psi(0)=0$ such that for all $\C \in \K$ with radius $\radius$,
		\begin{equation}\label{eq:decay} f(\xr(\C))-\qlbr(\C)\leq \psi(\radius) ,  \end{equation}
		Then  Algorithm~\ref{alg} terminates after a finite number of steps, and returns an $\epsilon$-optimal solution $\xbest$. 
	\end{enumerate}
\end{theorem}
We note that if a qBnB algorithm has convergence order $\alpha$ for some positive $\alpha$, then \eqref{eq:decay} holds. 
\begin{proof}
	We first show that for all $g$ visited by the algorithm, once the construction of the list $L_g$ is terminated (line~\ref{g_update}), the list contains all global minimizers of $f$ in $\C_0$.  We prove the claim by induction on $g$. For $g=0$ each minimizer is contained in $\C_0\in L_0$. Now assume the claim holds for some $g \in \NN $.  Let $x_*$ be  a minimizer, then by assumption there exists $\C \in L_g$ such that  $x_* \in \C$. The cube $\C$ will not be eliminated by the condition in Line~\ref{ln:eliminate} of Algorithm~\ref{alg}, since $q=\qlbr(\C) $ is a lower bound for $f_*$, and so will be partitioned into two cubes $\C_1,\C_2$ which will be placed in $L_{g+1}$. One of these cubes will contain $x_*$.
	
	We now prove the first claim of the theorem. Since $f$ is continuous and $\C_0$ is compact, there exists a minimizer $x_*$ of $f$ in $\C_0$.  For all $g$ visited by the algorithm, the list $\L_g$ contains a cube $\C_*$ which contains $x_*$ by the previous claim. We know that $\qlb(\C_*) \leq f_*$ and so by the definition of $\lb$ in Line~\ref{lb:def} as the minimum of all quasi-lower bounds in the generation $g$ we have that $\lb\leq f_*$. 
	
	We now prove the second claim. Note that if $\C \in L_g$ and $\C_1$ is a subcube obtained from $\C$ by subdivision along the longest edge of $\C$, then the radius of $\C_1$ is smaller than the radius of $\C$ by a multiplicative factor of at least 
	$$q=\left(\frac{d-3/4}{d}\right)^{1/2}<1 $$
	There exists $\delta>0$ such that $\psi(t)<\epsilon/2 $ for all $t, 0\leq t\leq \delta$. 
	For $g$ large enough $ \radius(\C_0) q^g<\delta $  and so all cubes in $L_g$ will have radius smaller than $\delta$. Let $\xbest$ be the point attained by the algorithm at the end of the $g$-th generation, and let  $\C_*$ be a cube in $L_g$ which contains a minimizer  of $f$. Then since $\qlb(\C_*)$ is a lower bound for $f_*$, and $f(\xbest)\leq f(\xr(\C_*)) $,
	\begin{equation}\label{eq:upper}
	\ub-f_*=f(\xbest)-f_* \leq f(\xr(\C_*))-f_*\leq f(\xr(\C_*))-\qlbr(\C_*)\leq \psi(\radius)<\epsilon/2 ,\end{equation}
	where $\radius$ is the radius of the cube $\C_*$.
	On the other hand for any $\C \in L_g$ with radius $r$, 
	$$f_*-\qlbr(\C)\leq f(\xr(\C))-\qlbr(\C)\leq \psi(\radius)<\epsilon/2$$
	and so taking the minimum over all $\C \in L_g$ we have
	\begin{equation}\label{eq:lower} 
	f_*-\lb<\epsilon/2 
	\end{equation}
	The combination of \eqref{eq:upper} and \eqref{eq:lower} shows that the algorithm terminates at the end of the $g$-th generation (if not beforehand) as $\ub-\lb<\epsilon $. When the stopping condition is met $\xbest$ is $\epsilon/2$-optimal by \eqref{eq:upper}.
\end{proof}

\section{Second order quasi-BnB for conditionally Lipschitz differentiable functions}\label{app:conditional_second_order}
In this appendix we review the quasi-BnB algorithm suggested in \cite{dym2019linearly}. This paper studies optimization problems of the form
\begin{equation}\label{eq:E}
\min_{x \in \C_0, y \in Y} E(x,y)
\end{equation}
where $E$ is continuous, $\C_0 \subseteq \RR^d$ is a cube, and  $Y$ is some compact subset of $\RR^n$. We make the following assumptions
\begin{enumerate}
	\item For fixed $x$, the function $y \mapsto E(x,y) $ can be minimized in polynomial time. 
	\item For fixed $y$, the function $x \mapsto E(x,y) $ is in $C^1(\C_0) $, with Lipschitz continuous gradients bounded by a Lipschitz constant $L_2$ which is independent of $y \in Y$. 
\end{enumerate}

There are several well known algorithmic problems with have this structure. Such problems are typically solved using an alternating minimization algorithm. Prominent examples being Expectation Maximization (EM) for parameter estimation \cite{moon1996expectation}, $k$-means for clustering \cite{jain2010data} and Iterative Closest Point (ICP) for rigid alignment \cite{besl1992method}. Our aim is to solve these optimization problems globally. To so we define the function 
$$f(x)=\min_{y \in Y} E(x,y) ,$$
so that the original optimization problem \eqref{eq:E} is equivalent to optimizing the continuous function $f$ over the cube $\C$. This formulation is particularly useful for problems like the rigid alignment problem, which was the focus of \cite{dym2019linearly},  where $d$ is small.

We now show how to define a quasi-lower bound rule for minimizing $f$. We claim that if we have for each fixed $y$ a quasi lower bound rule $\qlbr_y(\cdot) $ for minimizing the function $x \mapsto E(x,y) $, then
\begin{equation}\label{eq:qlb_accum} \qlbr(\C)=\inf_{y \in Y} \qlbr_y(\C) \end{equation}
is a valid quasi-lower bound rule for minimizing $f$ over $\C_0$. This is because 
for every cube $\C$ which contains a minimizer $x_*$ of $f$, there exists some $y_*$ such that  $(x_*,y_*)$ is  a minimizer of $E$, and so $x \mapsto E(x,y_*) $ satisfies
$$f(x_*)= E(x_*,y_*)\geq \qlbr_{y_*}(\C) \geq \min_{y \in Y} \qlbr_{y}(\C).$$
If we know that for any fixed $y \in \Y$ the optimization of $E(\cdot,y)$ over $\C_0$ is an unconstrained optimization problem, then using the quasi lower bound rule of qBnB(2) with  the uniform  Lipschitz gradient constant $L_2$ we obtain for cubes $\C \in \K$ with center $x_\C$ and radius $\radius$,
$$\qlbr(\C)=\min_{y \in Y} E(x_\C,y)-\frac{L_2}{2}\radius^2=F(x_\C)-\frac{L_2}{2}\radius^2 . $$
Note that in the general case  we can get an analogous quasi lower bound by using the quasi-lower bound rule of constrained-qBnB(2). 

We note that \eqref{eq:qlb_accum} can be used to derive other (quasi)-lower bounds from a family of (quasi)-lower bound rules $\qlbr_y, y\in Y$, and indeed this have been suggested using Lipschitz first order bounds \cite{yang2015go,pfeuffer2012discrete,hartley2009global}. However, bounding $\qlb_y(\C) $ uniformly in $y$  for second order lower bounds discussed in Subsection~\ref{subsec:BnB} seems to be a formidable task. Indeed to the best of our knowledge the method suggested in \cite{dym2019linearly} is the only method with second order convergence for this kind of problem. 

\newpage

%
%
%
\end{document}